\documentclass[10pt,a4paper]{article}
\usepackage{amsmath}
\usepackage{amsfonts}
\usepackage{amssymb}
\usepackage{graphicx}
\usepackage{epstopdf}
\usepackage[pagewise]{lineno}

\usepackage[backend=bibtex,style=alphabetic,backref=false,doi=false,
isbn=false,url=false]{biblatex}
\bibliography{bipartition_polynomial.bib} 

\newtheorem{theorem}{Theorem}

\newtheorem{corollary}[theorem]{Corollary}

\newtheorem{lemma}[theorem]{Lemma}
\newtheorem{proposition}[theorem]{Proposition}
\newtheorem{remark}[theorem]{Remark}
\newenvironment{proof}[1][Proof]{\noindent\textbf{#1.} }{\ \rule{0.5em}{0.5em}}

\usepackage{xspace}
\newcommand{\cp}{\text{Comp}} 
\newcommand{\iso}{\text{iso}} 
\newcommand{\bprecon}{bp-re\-con\-struc\-ti\-ble\xspace}
\newcommand{\pdeck}{poly\-no\-mi\-al-deck\xspace}
\newcommand{\sdeck}{$*$-deck\xspace}

\title{The Bipartition Polynomial of a Graph: Reconstruction, Decomposition, and
	Applications}
\author{Seongmin Ok \\
	Korea Institute for Advanced Study, Seoul
	\and Peter Tittmann\\
University of Applied Sciences Mittweida}

\begin{document}
\maketitle

\begin{abstract}
	The bipartition polynomial of a graph, introduced in \cite{Dod2015}, is a
	generalization of many other graph polynomials, including the domination, Ising, 
	matching, independence, cut, and Euler polynomial. We show in this paper that it
	is also a powerful tool for proving graph properties.
	In addition, we can show that the bipartition polynomial is polynomially reconstructible,
	which means that we can recover it from the multiset of bipartition polynomials
	of one-edge-deleted subgraphs. 
\end{abstract}

\section{Introduction}
The \emph{bipartition polynomial} $B(G;x,y,z)$ of a simple graph $G$ has been introduced 
in \cite{Dod2015}. The bipartition polynomial is related to the set of bipartite subgraphs 
of $G$; it generalizes the Ising polynomial \cite{Andren2009}, the matching polynomial
\cite{Farrell1979}, the independence polynomial (in case of regular graphs) 
\cite{Levit2005,Gutman1983}, the domination polynomial \cite{Arocha2000}, 
the Eulerian subgraph polynomial \cite{Aigner2007}, and the cut polynomial of a graph.
In this paper, we consider the natural generalization of the bipartition polynomial to 
graphs with parallel edges.

Let $G=(V,E)$ be a simple undirected graph with vertex set $V$ and edge set $E$. The
open neighborhood of a vertex $v$ of $G$ is denoted by $N(v)$ or $N_G(v)$. It is the
set of all vertices of $G$ that are adjacent to $v$. The closed neighborhood of $v$ is
defined by $N_G[v]=N_G(v)\cup \{v\}$. The neighborhood of a vertex subset $W \subseteq V$ is:
\begin{align*}
N_G(W) &= \bigcup_{w\in W}N_G(w) \setminus W, \\
N_G[W] &= N_G(W) \cup W.
\end{align*}
The \emph{edge boundary} $\partial W$ of a vertex subset $W$ of $G$ is
\[ \partial W = \{\{u,v \} \mid u\in W\text{ and }v\in V\setminus W \}, \]
i.e., the set of all edges of $G$ with exactly one end vertex in $W$. Throughout this
paper, we denote by $n$ the order, $n=|V|$, by $m$ the size, $m=|E|$,
and by $k(G)$ the number of components of $G$.

The bipartition polynomial of a graph $G$ is defined by
\begin{equation}
B(G;x,y,z)=\sum_{W\subseteq V} x^{|W|}\sum_{F\subseteq \partial W} y^{|N_{(V,F)}(W)|}
z^{|F|}.
\label{eq:bipart_def}
\end{equation}
Note that the definitions of neighborhood, edge boundary and Equation \eqref{eq:bipart_def} 
can be easily extended to graphs with parallel edges. From now on, unless otherwise stated, 
we allow graphs to have parallel edges. Note that adding loops does not change the bipartition polynomial.

\begin{figure}[ht]
	\centering
	\includegraphics[scale=0.8]{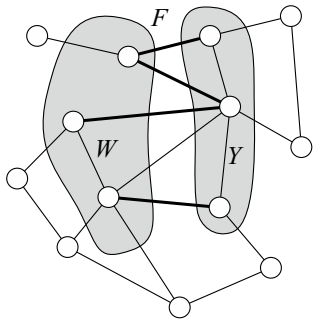}
	\caption{Illustration of the definition of the bipartition polynomial}
	\label{fig:bipart_def}
\end{figure}

Figure \ref{fig:bipart_def} provides an illustration of the definition given in Equation
(\ref{eq:bipart_def}). First we select a vertex subset $W$, which is located within the
left gray-shaded bubble in Figure \ref{fig:bipart_def}. The cardinality of the set $W$
is counted in the exponent of the variable $x$. The edge boundary $\partial W$ consists of all
edges that stick out from that bubble. Assume we select the edges shown in bold as subset
$F\subseteq \partial W$. The end vertices of these edges outside $W$ are presented within
the next bubble that is labeled with $Y$. The cardinality of the set $Y$ is counted in the 
exponent of variable $y$ of the bipartition polynomial. The third variable, $z$, counts the
edges in $F$. We see that by the definition of $F$
always a bipartite subgraph of $G$ is defined, which is the reason for the naming  
`bipartition polynomial'. If $H=(S\cup T,F)$ is a connected bipartite graph, then the partition
sets $S$ and $T$ are uniquely defined (up to order). 

Equation (\ref{eq:bipart_def}) implies that we can derive the order and size of a graph
from its bipartition polynomial:
\begin{eqnarray}
	 n = \deg B(G;x,1,1) \label{eq:n} \\ 
	 m = \frac{1}{2}[xyz]B(G;x,y,z), \label{eq:m}
\end{eqnarray}
where $[xyz]B(G;x,y,z)$ denotes the coefficient of $xyz$ in $B$. 
\begin{proposition}\label{prop:bipart}
	A loopless graph $G$ is bipartite if and only if
	\[ \frac{1}{2}[xyz]B(G;x,y,z) = \deg B(G;1,1,z). \]
\end{proposition}
\begin{proof}
	The left-hand side is, according to Equation (\ref{eq:m}), the number of edges of $G$.
	A graph $G=(V,E)$ is bipartite if and only if there is a vertex subset $W\subseteq V$
	with $\partial W = E$. Equation (\ref{eq:bipart_def}) shows that in this case the 
	degree of $z$ in $B(G;x,y,z)$ is equal to $m$. 
\end{proof}

In the remaining part of this paper, we present different representation and decompositions 
of the bipartition polynomial (Section \ref{sect:decom}), derive relations to other
graph polynomials (Section \ref{sec:poly}), prove its polynomial reconstructibility
(Section \ref{sec:recon}), and provide some applications for proving graph properties
(Section \ref{sec:appl}).

\section{Representations and Decomposition}\label{sect:decom}

In this section, we provide some different representations of the bipartition polynomial
and decomposition formulae with respect to vertex and edge deletions.

\subsection{Representations of the Bipartition Polynomial}

\begin{theorem}[product representation, \cite{Dod2015}]\label{theo:prod_representation}
	The bipartition polynomial of a graph $G$ can be represented as
	\begin{equation}
		B(G;x,y,z) = \sum_{W\subseteq V} x^{|W|} \prod_{v\in N_{G}(W)}
		\left[y\left[(1+z)^{|\partial v \cap \partial W|}-1 \right]  +1 \right].
		\label{eq:prod_representation}
	\end{equation}
	The bipartition polynomial of a simple graph $G=(V,E)$ satisfies
	\begin{equation}
		B(G;x,y,z) = \sum_{W\subseteq V} x^{|W|} \prod_{v\in N_{G}(W)}
		\left[y\left[(1+z)^{|N_{G}(v)\cap W|}-1 \right]  +1 \right].
		\label{eq:prod_representation_simple}
	\end{equation}
\end{theorem}
\begin{corollary}\label{coro:k(G)}
	The number of components of a graph $G$ is $\log_{2}B(G;1,1,-1)$.
\end{corollary}
\begin{proof}
	From Equation (\ref{eq:prod_representation}), we obtain
	\[
	B(G;1,1,-1) = \sum_{W\subseteq V}\prod_{v\in N_{G}(W)}
	0^{\left\vert \partial v \cap \partial W\right\vert }.
	\]
	The product vanishes for all $W\subseteq V$ with $N_{G}(W)\neq\emptyset$,
	since $\partial v \cap \partial W\neq\emptyset$ for all $v\in N_{G}(W)$. The product
	equals 1 if $N_{G}(W)$ is empty. We have $N_{G}(W)=\emptyset$ if and only if
	$W$ is the (possibly empty) union of vertex sets of components of $G$. For a
	graph with $k$ components, there are $2^{k}$ ways to form a union of the
	vertex sets of the components. Hence we obtain $B(G;1,1,-1)=2^{k}$.
\end{proof}

The proof of the last proposition also yields the following statement.
\begin{corollary}\label{coro:comp_sizes}
	If $G$ consists of $k$ components $G_{1},...,G_{k}$ such that the order of
	$G_{i}$ is $k_{i}$, then%
	\[
	B(G;x,1,-1)=\prod\limits_{i=1}^{k}(1+x^{k_{i}}).
	\]
\end{corollary}
Consequently, we can derive the order of all components of $G$ by the
following simple procedure. The order of the first component is the smallest
positive power, say $k_{1}$, of $x$ in $B(G;x,1,-1)$. Now divide $B(G;x,1,-1)
$ by $(1+x^{k_{1}})$ and proceed step by step with the resulting polynomial in
the same manner until you obtain the constant polynomial 1.

A connected bipartite graph with at least one edge is called \emph{proper}.
For any given graph $G$, we denote by $\cp(G)$ the set of proper components of $G$. 
As an abbreviation, we use $\cp(V,E)$ instead of $\cp((V,E))$.
The number of isolated vertices of a graph $G=(V,E)$ is denoted by $\iso(G)$ or
by $\iso(V,E)$.
\begin{theorem}[bipartite representation, \cite{Dod2015}]\label{theo:bip_representation}
	The bipartition polynomial of a graph $G=(V,E)$ satisfies
	\begin{equation}
		B(G;x,y,z) = \hspace{-14pt}\sum_{\substack{F\subseteq E \\(V,F)\text{ is bipartite}}}
		\hspace{-20pt} z^{|F|}(1+x)^{\mathrm{iso}(V,F)} 
		\hspace{-20pt} \prod_{(S\cup T,A)\in\cp(V,F)}
		\hspace{-25pt}(x^{|S|}y^{|T|}+x^{|T|}y^{|S|}).
	\end{equation}  
\end{theorem}

For another representation of the bipartition polynomial using so-called \emph{activity}, 
we assume that the edge set $E=\{e_1,\ldots,e_m\}$ of the graph $G=(V,E)$ is linearly 
ordered, that is
$e_1<e_2<\cdots <e_m$. Let $H$ be a \emph{spanning forest} of $G$, which is 
a forest $H=(V,F)$ with the same vertex set as $G$ and $F\subseteq E$. An edge 
$e\in E\setminus F$ is \emph{externally active} with respect to the forest $H$ if
it is the largest edge in a cycle of even length of $H+e$. We denote by 
$\mathrm{ext}(H)$ the number of externally active edges of $H$. Note that our definition 
of external activity is little different from that of Tutte \cite{Tutte1954}.
\begin{theorem}[forest representation, \cite{Dod2015}]\label{theo:forest_representation}
	The bipartition polynomial of a graph $G=(V,E)$ satisfies
	\begin{equation}
		\begin{array}{l}
			B(G;x,y,z) = {\displaystyle\sum_{\substack{H \text{ is spanning}\\ \text{ forest of }G }}
				\hspace{-10pt} (1+x)^{\mathrm{iso}(H)} z^{n-k(H)}
				(1+z)^{\mathrm{ext}(H)}}  \\
			{\displaystyle\hspace{70pt}\times \hspace{-10pt} \prod_{(S\cup T,A)\in\mathrm{Comp}(H)}
				\hspace{-25pt}(x^{|S|}y^{|T|}+x^{|T|}y^{|S|})}.
		\end{array}
	\end{equation}	
\end{theorem}

\begin{remark}
	In \cite{Dod2015}, the Theorems 2,3, and 4 are proven for simple graphs only.
	However, the generalization to non-simple graphs is straightforward.
\end{remark}
We need also the following result, which is proven in \cite{Dod2015} too.
\begin{theorem}\label{theo:bip_multiplicative}
	Let $G$ be a graph consisting of $k$ components $G_1,\ldots,G_k$. Then
	\[ 
		B(G;x,y,z) = \prod_{i=1}^{k} B(G_i;x,y,z).
	\]
\end{theorem}

\subsection{Vertex and Edge Deletion}
First we consider decompositions for the bipartition polynomial of a graph with respect
to local vertex and edge operations.
\begin{theorem}
	The bipartition polynomial of a graph $G=(V,E)$ satisfies for each vertex $v\in V$
	the relation
	\begin{align*}
	B(G;x,y,z) &= (1+x) B(G-v;x,y,z) \\
	& + \hspace{-10pt}\sum_{\substack{(S\cup T,F) \text{ conn. bip.}\\v\in S\cup T}}
	\hspace{-20pt} z^{|F|}(x^{|S|}y^{|T|}+x^{|T|}y^{|S|})B(G-(S\cup T);x,y,z) \\
	&= (1+x) B(G-v;x,y,z) 
	+ \hspace{-10pt}\sum_{\substack{(S\cup T,F) \text{ tree of } G\\v\in S\cup T}}
	\hspace{-20pt} z^{|F|-1}(1+z)^{\mathrm{ext}(S\cup T,F)} \\
	& \times (x^{|S|}y^{|T|}+x^{|T|}y^{|S|})B(G-(S\cup T);x,y,z),
	\end{align*}
	where the first sum is taken over all proper subgraphs, and the second sum is over 
	all nontrivial trees (having at least one edge) of $G$ that contain the vertex $v$.
\end{theorem}
\begin{proof}
We show the proof for the first equality and the second one can be shown similarly. From 
Theorem \ref{theo:bip_representation}, we obtain
	\begin{align*}
		B(G;x,y,z) &= \hspace{-14pt}\sum_{\substack{F\subseteq E \\(V,F)\text{ is bipartite}}}
		\hspace{-20pt} z^{|F|}(1+x)^{\mathrm{iso}(V,F)} 
		\hspace{-20pt} \prod_{(S\cup T,A)\in\mathrm{Comp}(V,F)}
		\hspace{-25pt}(x^{|S|}y^{|T|}+x^{|T|}y^{|S|}) \\
		 &= \hspace{-14pt}\sum_{\substack{F\subseteq E\setminus \partial v 
		 		\\(V,F)\text{ is bipartite}}}
		 \hspace{-20pt} z^{|F|}(1+x)^{\mathrm{iso}(V,F)} 
		 \hspace{-20pt} \prod_{(S\cup T,A)\in\mathrm{Comp}(V,F)}
		 \hspace{-25pt}(x^{|S|}y^{|T|}+x^{|T|}y^{|S|}) \\
		 &+ \hspace{-14pt}\sum_{\substack{F\subseteq E,\, F\cap \partial v\neq \emptyset 
		 		\\(V,F)\text{ is bipartite}}}
		 \hspace{-20pt} z^{|F|}(1+x)^{\mathrm{iso}(V,F)} 
		 \hspace{-20pt} \prod_{(S\cup T,A)\in\mathrm{Comp}(V,F)}
		 \hspace{-25pt}(x^{|S|}y^{|T|}+x^{|T|}y^{|S|}) \\
		&= (1+x) B(G-v;x,y,z) \\
		& + \hspace{-10pt}\sum_{\substack{(S\cup T,F) \text{ conn. bip.}\\v\in S\cup T}}
		\hspace{-20pt} z^{|F|}(x^{|S|}y^{|T|}+x^{|T|}y^{|S|})B(G-(S\cup T);x,y,z) \\
	\end{align*}
The last equality results from factoring out the term of the product that corresponds
to the component containing $v$ and applying Theorem \ref{theo:bip_multiplicative}.
\end{proof}

The proof of the next statement can be performed in the same way.
\begin{theorem}
	Let $G=(V,E)$ be a graph and $e\in E$; then
	\begin{align*}
	B(G;x,y,z) &= B(G-e;x,y,z) \\
	& + \hspace{-10pt}\sum_{\substack{(S\cup T,F) \text{ conn. bip.}\\e\in F}}
	\hspace{-20pt} z^{|F|}(x^{|S|}y^{|T|}+x^{|T|}y^{|S|})B(G-(S\cup T);x,y,z).
	\end{align*}
\end{theorem}

\section{Graph Polynomials that can be Derived from the Bipartition Polynomial}
\label{sec:poly}

Several well-known graph polynomials can be obtained by substitution of the variables
of the bipartition polynomial and (in some case) by multiplication with a certain
factor that can easily be obtained from graph parameters like order, size, and the number
of components. First we recall some results from \cite{Dod2015}.

\paragraph{Domination polynomial}
The \emph{domination polynomial} of a graph $G=(V,E)$, introduced in \cite{Arocha2000}, 
is the ordinary generating function
for the number of dominating sets of $G$. Let $d_k(G)$ be the number of dominating sets 
of size $k$ of $G$. We define the domination polynomial of $G$ by
\[ 
D(G,x) = \sum_{k=0}^{n}d_k(G)x^k.
\]
The domination polynomial satisfies, \cite{Dod2015},
\begin{equation}
	D(G,x) =(1+x)^{n}B\left(G;\frac{-1}{1+x},\frac{x}{1+x},-1\right).
	\label{eq:dom_bip}
\end{equation}
A \emph{generalized domination polynomial} is given by
\[ 
B(G;x,1-y,-1) = \sum_{W\subseteq V} x^{|W|}y^{|N_G(W)|},
\]
which follows directly from Theorem \ref{theo:prod_representation}. The variable $y$
counts here the number of vertices that are dominated by a given set $W$. Consequently, 
the coefficient of $x^iy^j$ in $B(G;x,1-y,-1)$ gives the number of vertex subsets of
cardinality $i$ that dominate a vertex set of size $j$.

\paragraph{Ising polynomial}
The \emph{Ising polynomial} of a graph $G$ is defined by
\[ 
Z(G;x,y) = x^ny^m\sum_{W\subseteq V}x^{-|W|}y^{-|\partial W|}.	
\]
The Ising polynomial has been differently introduced in \cite{Andren2009} by
\[ 
\tilde{Z}(G;x,y) = \sum_{\sigma\in \Omega}x^{\epsilon(\sigma)}y^{M(\sigma)},
\]
where $\sigma : V\rightarrow \{-1,1\}$ is the \emph{state} of $G=(V,E)$,  $\sigma(v)$ 
the \emph{magnetization} of $v\in V$, $\Omega$ the set of all states of $G$. The sum
\[
M(\sigma)=\sum\limits_{v\in V}\sigma(v) 
\]
is called \emph{magnetization} of $G$ with respect to $\sigma$. The parameter 
$\epsilon(\sigma,e)$ defines the \emph{energy} of the edge $e\in E$. The energy of $G$ 
with respect to $\sigma$ is 
\[ 
\epsilon(\sigma)=\sum\limits_{e\in E}\epsilon(\sigma,e).
\]
The here given notions result from the interpretation of the Ising polynomial (Ising model)
in statistical physics. The relation between the two above given representations of the Ising
polynomial is
\[
\tilde{Z}(G;x,y) = x^{-n}y^{-m}Z(G;x^2,y^2).
\]
Generalizations and modifications of the Ising polynomial and their efficient
computation in graphs of bounded clique-width are considered in \cite{Kotek2015}.

The Ising polynomial can be obtained from the bipartition polynomial by, \cite{Dod2015},
\begin{equation}
	Z(G;x,y) = x^{n}y^{m}B\left(G;\frac{1}{x},1,\frac{1}{y}-1\right). 
	\label{eq:Ising_bip}
\end{equation}

\paragraph{Cut polynomial}
The \emph{cut polynomial} of a graph $G=(V,E)$ is the ordinary generating function 
for the number of cuts of $G$,
\[ 
C(G,z) = \frac{1}{2^{k(G)}}\sum_{W\in V}z^{|\partial W|}.
\]
The relation between cut polynomial and bipartition polynomial is given by, 
see \cite{Dod2015},
\begin{equation}
	C(G,z) = \frac{1}{2^{k(G)}} B(G;1,1,z-1).
	\label{eq:cut_bip}
\end{equation} 
\begin{corollary}\label{coro:forest}
	A graph $G$ of order $n$ with $k$ components is a forest if and only if 
	$C(G,z)=(1+z)^{n-k}$.
\end{corollary}
\begin{proof}
	The statement follows from the fact that a forest is the only graph for which any
	edge subset is a cut.
\end{proof}

The polynomial
\begin{equation}
	B(G;x,1,z-1) = \sum_{W \subseteq V}	x^{|W|}z^{|\partial W|}
	\label{eq:gen_cut}
\end{equation}
can be considered as a generalized cut polynomial; it is equivalent to the Ising polynomial.
Equation (\ref{eq:gen_cut}) implies also
\begin{equation} 
	\left. 		\frac{1}{\partial x}B(G;x,1,t-1) \right\vert_{x=0}
	= \sum_{v\in V} t^{\deg v},
	\label{eq:deg_gen}
\end{equation}
which is the \emph{degree generating function} of $G$.

\paragraph{Euler polynomial} 
An \emph{Eulerian subgraph} of a graph $G=(V,E)$ is a spanning subgraph 
of $G$ in which all vertices have even degree. The \emph{Euler polynomial}
of $G$ is defined by
\[ 
\mathcal{E}(G,z) = \sum_{\substack{F\subseteq E \\ (V,F)\text{ is Eulerian}}}
\hspace{-20pt}z^{|F|}.
\]
In \cite{Aigner2007} it is shown that the Euler polynomial is related to the Tutte 
polynomial via
\[ 
\mathcal{E}(G,z) = (1-z)^{m-n+k(G)}z^{n-k(G)}
T\left(G;\frac{1}{z},\frac{1+z}{1-z} \right).
\]
There is also a nice direct relation between cut polynomial and Euler polynomial,
which is also shown in \cite{Aigner2007},
\begin{equation}
	C(G,z) = \frac{(1+z)^{|E|}}{2^{|E|-|V|+k(G)}}
	\;\mathcal{E}\left(G,\frac{1-z}{1+z}\right)
	\label{eq:cut_Euler}
\end{equation}
Solving Equation (\ref{eq:cut_Euler}) for $\mathcal{E}(G,z)$ and substituting $C$
according to Equation (\ref{eq:cut_bip}) yields
\begin{equation} 
	\mathcal{E}(G,z) = \frac{(1+z)^{|E|}}{2^{|V|}}B\left(G;1,1,\frac{-2z}{1+z} \right).
	\label{eq:euler_bip}
\end{equation}
Let $G$ be a plane graph (a planar graph with a given embedding in the plane) and $G^*$ its
geometric dual. The set of cycles of $G$ is in one-to-one correspondence with the set
of cuts of $G^*$, which yields
\begin{equation*}
	\mathcal{E}(G,z) = C(G^*,z)
\end{equation*}
or, corresponding to Equations (\ref{eq:cut_bip}) and (\ref{eq:euler_bip}),
\begin{equation}
	(1+z)^m B\left(G;1,1,\frac{-2z}{1+z}\right) = 2^{n-1}B(G^*;1,1,z-1).
	\label{eq:planar}
\end{equation}

\paragraph{Van der Waerden polynomial}
The definition of this polynomial is presented in \cite{Andren2009} and based on an idea 
given in \cite{Waerden1941}.
Let $G=(V,E)$ be a graph of order $n$ and size $m$.
Let $w_{ij}(G)$ be the number of subgraphs of $G$ with exactly $j$ edges and $i$ vertices
of odd degree. The \emph{van der Waerden polynomial} of $G$ is defined by
\[ 
W(G;x,y) = \sum_{i=0}^{n}\sum_{j=0}^{m}w_{ij}(G)x^i y^j. 
\] 
From \cite{Andren2009} (Theorem 2.9), we obtain easily
\[ 
W(G;x,y) = \left(\frac{1-x}{2}\right)^n(1-y)^m 
Z\left(G;\frac{1+x}{1-y},\frac{1+y}{1-y}\right),
\]
where $Z$ is the Ising polynomial.
The van der Waerden polynomial can be derived from the bipartition polynomial by
\begin{equation}
	W(G;x,y) = \left(\frac{1+x}{2}\right)^n(1+y)^m 
	B\left(G;\frac{1-x}{1+x},1,-\frac{2y}{1+y}\right),
	\label{eq:vdwarden_bipart}
\end{equation}
where we use Equation (\ref{eq:Ising_bip}).

\paragraph{Matching polynomial}
The \emph{matching polynomial}, see \cite{Farrell1979}, of $G$ is defined by
\[ 
M(G,x) = \sum_{\substack{F\subseteq E\\ F\text{ matching in }G} }
\hspace{-20pt}x^{|F|}.
\]
Notice that the definition that is given here corresponds to \emph{matching generating
	polynomial} from \cite{Lovasz2009}.
A subgraph of $G$ with exactly $k$ edges and exactly $2k$ vertices of odd degree is
a matching in $G$, which yields
\[ 
M(G,t) = \lim_{y\rightarrow 0} W(G;t y^{-\frac{1}{2}},y).
\]
Substituting Equation (\ref{eq:vdwarden_bipart}) for $W$, we obtain
\begin{equation}
	M(G,t) = \lim_{y\rightarrow 0} \left(\frac{\sqrt{y}+t}{2\sqrt{y}}\right)^n
	(1+y)^m B\left(G;\frac{\sqrt{y}-t}{\sqrt{y}+t},1,-\frac{2y}{1+y}\right).
	\label{eq:match_bipart}
\end{equation}

\paragraph{Independence polynomial}
The \emph{independence polynomial} of a graph $G=(V,E)$ is the ordinary generating function
for the number of independent set of $G$,
\[ 
I(G,x) = \sum_{\substack{W\subseteq V\\ W\text{ independent in }G}}
\hspace{-25pt}x^{|W|}.
\]
If $G$ is a simple $r$-regular graph, then
\[
I(G,t) = \lim_{x\rightarrow0}B\left(G;tx^{r},1,\frac{1}{x}-1\right).
\]	
The proof of this relation is given in \cite{Dod2015}.
We can easily rewrite the last equation in order to avoid the limit:
\[ 
I(G,t) = \frac{1}{2\pi}\int_{0}^{2\pi}B(G;te^{irx},1,e^{-ix}-1)dx.
\]
The substitution $x\mapsto e^{ix}$ transforms each power of $x$ into a periodic function
whose period divides $2\pi$ such that the integration over $[0,2\pi]$ yields the constant
term (with respect to $x$) multiplied by $2\pi$.

\section{Polynomial Reconstruction} \label{sec:recon}

One of the important questions about graph polynomials is their distinguishing power, 
which can be stated as follows:

Let $\mathcal{C}$ be a graph class and let $P$ be a polynomial-valued isomorphism 
invariant defined on $\mathcal{C}$. Are there nonismorphic graphs $G$ and $H$ in 
$\mathcal{C}$ such that $P(G) = P(H)$?

Although we know that the bipartition polynomial cannot distinguish all graphs up to 
isomorphism, see \cite{Dod2015}, we do not know yet whether there are two nonisomorphic 
trees with the 
same bipartition polynomial. Instead, as trees are well-known to be `reconstructible' 
in various senses, we show that the bipartition polynomial of a graph is 
edge-reconstructible from its polynomial-deck, which shall be defined precisely below.

For a graph $G$, its \emph{\pdeck} is the multiset $\{B(G-e)\}_{e \in E(G)}$. We show 
that the bipartition polynomial is `edge-reconstructible' in most cases in the following 
sense:

A graph $G$ is \emph{\bprecon} if whenever a graph $H$ has the same \pdeck as $G$ we 
have $B(H) = B(G)$.

Unfortunately, there are some graphs with few edges that are not \bprecon. To describe 
such examples, let $P_s$ and $C_s$ denote respectively the path and cycle on $s$ 
vertices. We denote by $C_s + tP_1$ the disjoint union of $C_s$ and $t$ isolated 
vertices, and the graphs $P_s + tP_1$, $sP_2 + tP_1$ etc. are defined similarly. 
The following graphs in each line have the same \pdeck but have different bipartition 
polynomial.
\begin{itemize}
	\item $C_2 + (t+2)P_1$, $P_3 + (t+1)P_1$ and $2P_2 + tP_1$ for $t \geq 0$.
	\item $C_3 + (t+1)P_1$ and $K_{1,3} + tP_1$ for $t \geq 0$.
\end{itemize}

Note that the graphs on each line for fixed $t$ not only have the same \pdeck but 
also have the same collection of one-edge-deleted subgraphs.

We prove the following in this section.

\begin{theorem} \label{theo:reconstruct}
	A graph $G$ is \bprecon unless $G$ is one of the exceptions above. In particular, 
	all graphs with at least four edges are \bprecon.
\end{theorem}

\subsection{Graphs with Isolated Vertices}
We shall use the following information on graphs that are deducible from the bipartition 
polynomial. The statement combines the results given in Equations
(\ref{eq:n}), (\ref{eq:m}), (\ref{eq:deg_gen}), Proposition \ref{prop:bipart}, 
and Corollaries \ref{coro:comp_sizes}, \ref{coro:k(G)}, \ref{coro:forest}. 

\begin{theorem} \label{theo:bip-info-collection}
	Let $G$ be a graph. The bipartition polynomial of $G$ yields $|V(G)|$, $|E(G)|$, 
	$k(G)$, $\iso(G)$, the degree sequence, and the multiset of orders of all components of $G$. 
	We can also decide from $B(G)$ whether $G$ is bipartite, a forest, a path, or
	connected. (The last two properties follow from the other ones.)
\end{theorem}

We begin the proof of Theorem \ref{theo:reconstruct} with the case when two graphs 
$G$ and $H$ have different number of isolated vertices but the same \pdeck. Note that 
from Theorem \ref{theo:bip-info-collection}, we know that two graphs with a different 
number of isolated vertices have a different bipartition polynomial.

\begin{lemma} \label{lem:dif-isol}
	Let $G$ and $H$ be two graphs having different number of isolated vertices. 
	If $G$ and $H$ have the same \pdeck, then there exists $t \geq 0$ such that either
	\begin{itemize}
		\setlength{\itemsep}{0pt}
		\setlength{\parsep}{0pt}
		\setlength{\parskip}{0pt}
		
		\item $\{G, H\} \subset \{C_2 + (t+2)P_1, P_3 + (t+1)P_1, 2P_2 + tP_1 \}$ or
		\item $\{G, H\} = \{C_3 + (t+1)P_1, K_{1,3} + tP_1 \}$.
	\end{itemize}
\end{lemma}

\begin{proof}
	Suppose $G$ and $H$ have the same \pdeck and $\iso(G) = t$ while $\iso(H) > t$. 
	Since $\iso(G-e) \leq t+2$ for every edge $e \in E(G)$, we have $\iso(H) = t+1$ or $t+2$. 
	As $\iso(H-f) > t$ for all $f \in E(H)$, we have $\iso(G-e) > \iso(G)$ for all 
	$e \in E(G)$ implying that every edge of $G$ is incident with a vertex of degree 1. 
	That is, the components of $G$ are stars and isolated vertices.
	By Theorem \ref{theo:bip-info-collection}, we deduce that $H-f$ is a forest for every 
	$f \in E(H)$. Hence either $H$ itself is a forest or $H = C_s + t'P_1$ for some $s$ 
	and $t+1 \leq t' \leq t+2$.
	
	If $\iso(H) = t+2$ then every edge removal from $G$ produces two new isolated vertices, 
	so that $G = s'P_2 + tP_1$ for some $s'$. Moreover, no edge of $H$ is incident with 
	a vertex of degree 1, that is, $H = C_s + (t+2)P_1$. Since $G$ and $H$ have the same 
	order and an equal number of edges, we conclude $G = 2P_2 + tP_1$ and $H = C_2 + (t+2)P_1$.
	
	Now we assume $\iso(H) = t+1$. If $H = C_s + (t+1)P_1$ for some $s$, then for all $f \in E(H)$, $\iso(H-f) = t+1$ and $H-f$ has maximum degree at most two. As $G$ and $H$ have the same \pdeck, the same holds for $G-e$ for all $e \in E(G)$. Since $G$ is a disjoint union of stars with $t$ isolated vertices, the only possibility for this case is $G = K_{1,3} + tP_1$ and $H = C_3 + (t+1)P_1$.
	
	Now we also assume that $H$ is a forest. Theorem \ref{theo:bip-info-collection} states 
	that we can decide the orders of the components from the bipartition polynomial. 
	If $G-e$ for some $e \in E(G)$ has three $P_2$-components, then $H-f$ has it too for 
	some $f \in E(H)$ and $H$ has a $P_2$-component. Removing its edge produces a subgraph 
	with $t+3$ isolated vertices, which cannot be obtained from $G$ by removing only one 
	edge. Thus for all $e \in E(G)$, $G-e$ can have at most two $P_2$-components and $G$ 
	may have at most three $P_2$-components. If $G$ has three $P_2$-components, then they
	are the only nontrivial components of $G$.
	
	On the other hand, as $H$ is a forest, each nontrivial component of $H$ has at least 
	two leaves which leave $t+2$ isolated vertices each when removed. The number of leaves 
	of $H$ must be equal to the number of $P_2$-components of $G$, so that either 
	$G = 3P_2 + tP_1$ or $H = P_s + (t+1)P_1$ for some $s$. It is easy to check that 
	for this case the only possibility of non-isomorphic pair $G$ and $H$ with same \pdeck is 
	$G = 2P_2 + tP_1$ and $H = P_3 + (t+1)P_1$.
\end{proof}

\subsection{Cyclic Graphs}
Because of Lemma \ref{lem:dif-isol} we only need to compare those graphs without 
isolated vertices. The remaining part of our proof of Theorem \ref{theo:reconstruct} is 
presented in the following order.
\begin{enumerate}
	\item Every non-bipartite graph except $C_3 + tP_1$ for $t \geq 1$ is \bprecon.
	\item Every bipartite graph with a cycle except $C_2 + tP_1$ for $t \geq 2$ is \bprecon.
	\item Every forest except $P_3 + (t+1)P_1$, $2P_2 + tP_1$, $K_{1,3} + tP_1$ for $t \geq 0$ is \bprecon.
\end{enumerate}
The first two are simple but the proof for the third case is a bit longer so we defer it 
to Section \ref{subsec:forest-recon}.

Given a proper bipartite graph $K$ with bipartition $(U_1,U_2)$, let 
$m(K) = \min(|U_1|, |U_2|)$ and $M(K) = \max(|U_1|, |U_2|)$. 
Theorem \ref{theo:bip_representation} states 
\begin{equation}
	B(G;x,y,z) = \hspace{-15pt}\sum_{\substack{F \subseteq E \\ (V,F) \text{ bipartite}}}
	\hspace{-15pt} z^{|F|} (1+x)^{\iso(V,F)} 
	\hspace{-15pt} \prod_{K \in \cp(V,F)}
	\hspace{-10pt} [x^{M(K)} y^{m(K)} + x^{m(K)} y^{M(K)}].
	\label{eq:th5}
\end{equation}
Let us define for each $F \subseteq E$, a polynomial $\phi_G(F;x,y)$ or simply $\phi_G(F)$ as
\begin{equation*}
	\phi_G(F) := \left\lbrace \begin{array}{l}
		(1+x)^{\iso(V,F)} \hspace{-20pt}
		\prod\limits_{K \in \cp(V,F)} \hspace{-17pt}
		[x^{M(K)} y^{m(K)} + x^{m(K)} y^{M(K)}]\text{ if $(V,F)$ is bipartite,} \\
		0  \text{ otherwise.}
	\end{array}\right.
\end{equation*}
We also write $B(G)$ instead of $B(G;x,y,z)$ for convenience.
With this definition, Equation (\ref{eq:th5}) simplifies to
\begin{equation} \label{eq:biprep_short}
	B(G) = \sum_{F \subseteq E} z^{|F|} \phi_G(F).
\end{equation}
We consider the sum $B'(G) = \sum\limits_{e \in E} B(G-e)$ for a given multiset $\{ B(G-e)\}_{e \in E}$.
For each $F \subsetneq E$, the term $z^{|F|} \phi_G(F)$ appears precisely $|E| - |F|$ times on the right-hand-side so that
\begin{equation*}
	B'(G) = \sum_{F \subsetneq E} \big( |E| - |F| \big)  z^{|F|} \phi_G(F).
\end{equation*}
Thus for each $k = 0, 1, \ldots, |E|-1$, the coefficient of $z^k$ in $B(G)$ is the coefficient of $z^k$ in $B'(G)$ divided by $|E| - k$. The only remaining term to decide $B(G)$ is $z^{|E|} \phi_G(E)$. Therefore, to show $G$ is \bprecon it is enough to show that if $H = (V',E')$ is another graph with the same \pdeck as $G$, then $\phi_H(E') = \phi_G(E)$.

Now we show that nonbipartite graphs are \bprecon except $C_3 + tP_1$ for $t \geq 1$.
\begin{lemma} \label{lem:recon1}
	Every nonbipartite graph except $C_3 + tP_1$ for $t \geq 1$ is \bprecon.
\end{lemma}

\begin{proof}
	By Lemma \ref{lem:dif-isol}, it is enough to show that if $G$ is not bipartite, 
	$\iso(G) = \iso(H)$ and $H$ has the same \pdeck as $G$ then $B(G) = B(H)$. 
	We may additionally assume that $G$ and $H$ have no isolated vertices.
	
	Suppose $G = (V,E)$ is not bipartite, $\iso(G) = 0$ and let $D = \{B(G-e)\}_{e \in E}$. 
	Let $H = (V',E')$ be a graph with $\iso(H) = 0$ whose \pdeck is equal to $D$ as a 
	multiset. If $G-e$ is not bipartite for some $e \in E$ then from the corresponding 
	bipartition polynomial, we infer that $H-e'$ is nonbipartite for some $e' \in E'$ 
	and $\phi_H(E') = 0$, that is, $B(G) = B(H)$. If there is no such $e$, then $G$ is 
	an odd cycle. Applying Theorem \ref{theo:bip-info-collection} to $D$, we deduce that 
	every one-edge-deleted subgraph of $H$ is a path consisting of odd number of vertices 
	and the only graph with such property is an odd cycle, and hence $\phi_H(E') = 0$.
\end{proof}

We now suppose that $G = (V,E)$ is bipartite. Note that the degree of $\phi_G(F)$ is 
precisely $|V|$. Since
\[ 
	x^{M(K)}y^{m(K)} + x^{m(K)}y^{M(K)} 
	= (xy)^{m(K)} \left( x^{M(K) - m(K)} + y^{M(K) - m(K)} \right),
\]
to decide $\phi_G(E)$ we only need $M(K) - m(K)$ for each nontrivial component $K$ in $G$. 
If $G$ has $k$ nontrivial components with bipartitions $(U_i, V_i)$ for 
$i = 1, 2, \ldots, k$ and $t$ isolated vertices, then we say $G$ has \emph{type}
\[ 
	(a_1, a_2, \ldots, a_k, *, *, \ldots, *)
\]
where $a_i = \big||U_i| - |V_i| \big|$ and the number of $*$'s are $t$. We will ignore 
the order of entries in types. From now on we consider the type instead of $\phi_G(E)$.

\begin{lemma} \label{lem:recon2}
	Let $G$ be a bipartite graph. If $G$ has a cycle then $G$ is \bprecon unless $G = C_2 + tP_1$ for some $t \geq 2$.
\end{lemma}

\begin{proof}
	By Lemmas \ref{lem:dif-isol} and \ref{lem:recon1}, it is enough to show that if $H = (V',E')$ is another bipartite graph with the same \pdeck as $G$ and $\iso(G) =  \iso(H) = 0$, then the type of $H$ is uniquely determined. Note that the exceptions $C_2 + tP_1$ are automatically excluded since $t \geq 2$.
	
	If $G$ is connected, then $G-e$ is connected for some $e$, since $G$ has a cycle. Let 
	$e' \in E'$ be an edge such that $B(H-e') = B(G-e)$. We know that $H$ is bipartite and, 
	by Theorem \ref{theo:bip-info-collection}, $H-e'$ is connected. The coefficient of 
	$z^{|E'| - 1}$ in $B(H-e') = B(G-e)$ tells us the type of $H'$, which must be the same as 
	the type of $H$ and hence $\phi_H(E') = \phi_G(E)$.
	
	Suppose $G$ is not connected. Then $G-e$ contains a cycle for some $e \in E$, and by Theorem \ref{theo:bip-info-collection}, $H$ also has an edge $e'$ such that $H-e'$ contains a cycle. Thus $H$ has a cycle, and we choose $e'' \in E'$ such that $H-e''$ has minimum number of components among the one-edge-deleted subgraphs of $H$. The components of $H$ are vertex-wise same as the components of $H-e''$ and have precisely the same bipartitions. That is, the type of $H$ is the type of $H-e''$ which is again equal to the type of $G$. Hence $\phi_H(E') = \phi_G(E)$ and $G$ is \bprecon.
\end{proof}

\subsection{Bipartition Polynomials of Forests} \label{subsec:forest-recon}

In this section we prove the following lemma, thereby completing the proof of Theorem \ref{theo:reconstruct}.
\begin{lemma} \label{lem:recon3}
	Every forest except $2P_2 + tP_1$, $P_3 + (t+1)P_1$ and $K_{1,3} + tP_1$ for $t \geq 0$ is \bprecon.
\end{lemma}

To prove Lemma \ref{lem:recon3} for forests with at least four edges we show the following:

\begin{lemma} \label{lem:forestType}
	Let $F$ be a forest. The type of $F$ is uniquely determined from the degree sequence of $F$ and the multiset consisting of types of $F-e$ for all $e \in E(F)$.
\end{lemma}

In \cite{Delorme2002}, it was shown that if $G$ has at least four edges, then the degree 
sequence of $G$ is completely determined from the degree sequences of one-edge-deleted 
subgraphs. Theorem \ref{theo:bip-info-collection} states that the degree sequence is 
obtainable from the bipartition polynomial, so that Lemma \ref{lem:recon3} follows for 
forests with at least four edges. The missing cases for Lemma \ref{lem:recon3} without 
isolated vertices, $P_2$, $3P_2$, $P_3 + P_2$ and $P_4$ as simple graphs and also the 
non-simple ones are easy to check.

We shall use some lemmas about trees. For the definition of the type of a bipartite graph 
see the discussion preceding Lemma \ref{lem:recon2}. In a tree, 
a vertex of degree 1 is a \emph{leaf} and an edge incident with a leaf is a \emph{leaf-edge}. 
An edge is \emph{internal} if it is not a leaf-edge.

\begin{lemma} \label{lem:treetypes}
	Let $T$ be a tree with at least one edge. Let $(U,V)$ be the bipartition of $T$.
	\begin{enumerate}
		\renewcommand{\labelenumi}{(\roman{enumi})}
		\item If $U$ has all the leaves, then $|U| > |V|$.
		\item If $V$ has only one leaf, then $|U| \geq |V|$.
		\item If $T$ has type $(a)$ for $a \geq 1$, then $T$ has two edges $e_1, e_2$ such 
		that both $T-e_1$ and $T-e_2$ have type $(a-1,*)$.
		\item Suppose $T$ has type $(0)$. If $T-e$ has type $(1,1)$ for every internal edge 
		$e$, then the degrees of vertices of $T$ are all odd.
		\item Suppose $T$ has type $(2)$. If the types of $T-e$ with $*$ are all $(1,*)$, then
		either $T$ is $K_{1,3}$ or $T-f$ has type $(0,2)$ for some edge $f$.
	\end{enumerate}
\end{lemma}

\begin{proof}
	Let $u$ be a vertex in $U$. Consider $u$ as a root and direct every edge of $T$ away 
	from $u$. Then
	\[ 
		|U| = 1 + \sum_{v \in V(T)} d^+(v) = 1 + \sum_{v \in V} \big( d(v) - 1 \big) 
		= 1 + \sum_{v \in V} \big( d(v) - 2 \big) + |V|,
	\]
	so that
	\[ 
		|U| - |V| = 1 + \sum_{v \in V} \big( d(v) - 2 \big) \tag{*}
	\]
	and (i), (ii) follows immediately.
	
	Suppose $T$ has type $(a)$ for some $a \geq 1$. We may assume $|U| - |V| = a$. By (i) 
	and (ii), $U$ contains at least two leaves and removing their incident edges produce 
	forests, each of type $(a-1, *)$. Thus (iii) holds.
	
	Now we consider (iv). Suppose $T$ has type $(0)$ and $T-e$ has type $(1,1)$ for every 
	internal edge $e$ of $T$. If $T$ consists of only one edge then the conclusion holds. 
	Thus we assume that $T$ has at least one internal edge. Suppose that $T$ has a vertex 
	$v$ of even degree, for contradiction. Since $T$ has type $(0)$ $v$ is incident to at 
	least one internal edge. Let us say $v$ is adjacent to $s$ leaves and is incident 
	to $t$ internal edges $e_1, e_2, \ldots, e_t$ where $e_i = vu_i$ for 
	$i = 1, 2, \ldots, t$. Let us denote by $(U_i, V_i)$ the bipartition of the component 
	of $T-e_i$ containing $u_i$ such that $u_i \in U_i$. By the assumptions on $T$, 
	we know $|U_i| - |V_i|$ 
	is odd for all $i$. We consider the component of $T-e_1$ containing $v$. Its type is 
	given by
	\[ 
		\left\vert \sum_{i=2}^t |U_i| + s - \sum_{i=2}^t |V_i| - 1 \right\vert,
	\]
	which is an even number contradicting the assumption that $T-e_1$ has type $(1,1)$. 
	Hence (iv) follows.
	
	Lastly we show (v). Let $T$ be a tree with bipartition $(U,V)$ such that 
	$|U| - |V| = 2$. Suppose that for every leaf-edge $e$ of $T$,
	the forest $T-e$ has type $(1,*)$. That is, $U$ contains all the leaves. 
	From Equation $(*)$, we deduce that $V$ has 
	a unique vertex of degree 3 and all other vertices in $V$ have degree 2. If $V$ has 
	no vertex of degree 2 then $|V| = 1$ and $T$ is $K_{1,3}$. Suppose $V$ has a vertex, 
	say $v$, of degree 2. Let $e,f$ be the edges incident with $v$. If $e$ is a leaf-edge 
	then $T-f$ has type $(0,2)$. We assume that both $e$ and $f$ are internal edges. 
	The graph $T-e$ has two components, say $T_1 = (W_1, E_1)$ and $T_2 = (W_2, E_2)$ 
	where $f \in E_2$. Note that all the leaves of $T_1$ are in $U$ and all but one 
	leaves of $T_2$ are in $U$. By (i) and (ii), we have
	\[ 
		|U \cap W_1| > |V \cap W_1| \quad \text{ and } \quad |U \cap W_2| \geq |V \cap W_2|.
	\]
	Since $2 = |U| - |V| = (|U \cap W_1| - |V \cap W_1| ) + (|U \cap W_2| - |V \cap W_2| )$, 
	we have either
	\[ 
		|U \cap W_1| - |V \cap W_1| = 2 \text{ and } |U \cap W_2| - |V \cap W_2| = 0
	\]
	or
	\[ 
		|U \cap W_1| - |V \cap W_1| = 1 \text{ and } |U \cap W_2| - |V \cap W_2| = 1.
	\]
	For the former $T-e$ has type (0,2). For the latter $T-f$ has type (0,2). Thus (v) holds.
\end{proof}

\begin{proof}[Proof of Lemma \ref{lem:forestType}]
We shall divide the proof of Lemma \ref{lem:forestType} into three cases, depending on 
whether the forest $F$ has one component, two components or more than two components. 
In each case we show how to retrieve the type of $F$. We call the multiset of the types 
of $F-e$ for all edges $e$ as the \emph{type-deck} of $F$. The sub-multiset consisting of 
those types with $*$ is called \emph{\sdeck}. We can assume the following in all three 
cases. The reasoning is given below.
\begin{itemize}
	\setlength{\itemsep}{0pt}
	\setlength{\parsep}{0pt}
	\setlength{\parskip}{0pt}
	\item The types in the \sdeck contains precisely one $*$.
	\item If $F$ has more than one component, then at least one type in the \sdeck has a 
	zero as entry.
	\item No type in the \sdeck has more than one zero.
\end{itemize}
As the degree sequence of $F$ is given by the assumption, we may assume that $F$ has 
no isolated vertices. 
If the the \sdeck of $F$ contains a type with two $*$'s, then the two isolated vertices came 
from deleting an edge in a $P_2$-component of $F$, so that we can recover the type of $F$ 
by replacing the two $*$'s with a zero. Thus we may assume that the types in the \sdeck contains 
precisely one $*$.

Let $m$ be the minimum of integral entries in the types in the \sdeck. If $m>0$, then $F$ cannot 
induce a component of type $(m)$ by Lemma \ref{lem:treetypes} (iii). If $m \neq 1$, the entry 
$m$ is 
obtained from a component of type $(m+1)$, and the type of $F$ is obtained by replacing 
$(m,*)$ with a $m+1$. If $m =1$ and $F$ has more than one component, then $F$ cannot have 
a component of type $(0)$ so that again, by replacing $(m,*)$ with $m+1$ we retrieve the 
type of $F$. Hence we may assume that some types in the \sdeck have a zero.

Suppose that a type in the \sdeck has at least two zeros. Then $F$ has a component of type $(0)$. 
Among the types in the \sdeck, we choose one with a minimum number of zeros, denote this 
type by $X$. The zeros 
in $X$ came directly from $F$ and there must be a 1 and a $*$ which was obtained by removing 
a leaf-edge of a component with type $(0)$. Thus we replace $(1,*)$ by $(0)$ to retrieve the 
type of $F$. Now we may assume that no type in the \sdeck has more than one zero.

Now we prove Lemma \ref{lem:forestType}.

\textbf{Case 1.} $F$ is a tree.

If the the \sdeck has $(a,*)$ and $(a+2,*)$ for some $a$ then the only possible type of 
$F$ is $(a+1)$. Suppose $(a,*)$ is the unique type in the \sdeck. If $a \neq 1$ then by 
Lemma \ref{lem:treetypes} (iii) the type of $F$ is $(a+1)$. Suppose $(1,*)$ is the 
only type in the \sdeck. $F$ has two possibilities: $(0)$ and $(2)$. If the type-deck has $(0,2)$ 
then clearly $(2)$ is the case. If the degree sequence of $F$ is (3,1,1,1) then $F$ 
is $K_{1,3}$. If the type-deck does not have $(0,2)$ and $F$ is not $K_{1,3}$ then by 
Lemma \ref{lem:treetypes} (v) the type of $F$ is (0). This completes Case 1.

\textbf{Case 2.} $F$ has precisely two components.
By the assumptions the \sdeck has $(0,a,*)$ for some $a \geq 1$. If the \sdeck has $(0,a+2,*)$ 
also then $F$ has type $(0,a+1)$. Thus we assume $(0,a,*)$ is the unique type in the \sdeck with 
a 0. First we assume $a \geq 2$ and then consider the case $a=1$.

Suppose $a \geq 2$. If $F$ had type $(0,a-1)$, then by Lemma \ref{lem:treetypes} (iii) 
its \sdeck must have $(0,a-2,*)$ which is a contradiction. Thus $F$ has type $(0,a+1)$ 
or $(1,a)$.
If the \sdeck has $(1,a-1,*)$ then the type of $F$ cannot be $(0,a+1)$ and hence it is $(1,a)$.
If $F$ has type $(1,a)$ by Lemma \ref{lem:treetypes} (iii) the \sdeck has $(1,a-1,*)$. 
That is, $F$ has type $(1,a)$ if and only if the \sdeck has $(1,a-1,*)$, implying that we 
can decide the type of $F$ from its \sdeck.

Now we assume $a=1$ and the types in the \sdeck with a 0 are $(0,1,*)$. $F$ has one of three 
types: $(0,0)$, $(1,1)$ and $(0,2)$. If the type-deck has $(0,0,2)$ then $F$ has type $(0,2)$. 
Suppose the type-deck does not have $(0,0,2)$. If $F$ has type $(0,2)$, then every leaf-edge of 
a component of type $(2)$ produces $(1,*)$ when deleted. By Lemma \ref{lem:treetypes} 
(v), the component is $K_{1,3}$ and the \sdeck of $F$ has precisely three $(0,1,*)$. 
On the other hand, if $F$ had type $(1,1)$ then by Lemma \ref{lem:treetypes} (iii) there 
are at least four $(0,1,*)$ in the \sdeck of $F$. If $F$ had type $(0,0)$ then its \sdeck also 
have at least four $(0,1,*)$ since we assumed none of the components are $K_2$ and every 
tree has at least two leaves. Thus $(0,1,*)$ appears precisely three times in the \sdeck if 
and only if $F$ has type $(0,2)$.

Now we assume the \sdeck has at least four times $(0,1,*)$. $F$ has type $(0,0)$ or $(1,1)$. 
We may assume:
\begin{enumerate}
	\renewcommand{\labelenumi}{(\roman{enumi})}
	\item the \sdeck has only $(0,1,*)$.
	\item the type-deck consists of precisely $(0,1,*)$ and $(0,1,1)$. 
\end{enumerate}
The first assumption is because the only other possible type in the \sdeck is $(1,2,*)$, in 
which case $F$ has type $(1,1)$. For the second assertion, recall that we assumed that 
no component is $P_2$. Thus a component of type $(0)$ has an internal edge and by removing 
an internal edge from $(0,0)$ we get $(0,a,a)$. If $a \neq 1$ then $F$ has type $(0,0)$.

If $F$ has type $(1,1)$, then assumption (i) above implies for both components, one part 
of the bipartition contains all the leaves. From the proof of Lemma \ref{lem:treetypes} (i), 
all vertices in the other part have degree 2.

On the other hand, if $F$ has type $(0,0)$, then assumption (ii) above implies that for each component of $F$, every internal edge produces a forest of type $(1,1)$ when removed. Lemma \ref{lem:treetypes} (iv) asserts that all vertices of $F$ have odd degree.

Therefore, $F$ has type $(1,1)$ if $F$ has a vertex of degree 2, and $F$ has type $(0,0)$ 
otherwise. That is, we can determine the type of $F$ given all the assumptions so far. 
This completes Case 2.

\textbf{Case 3.} $F$ has more than two components.
By the assumptions after Lemma \ref{lem:treetypes} the \sdeck has a type $(0,a_1, a_2, \ldots, a_k, *)$ where $a_i > 0$ for all $i$. The question is to decide whether $F$ has a component of type $(0)$ or not. If it has, then the $(0)$ component is unique and the \sdeck has a type without zero. We replace $(1,*)$ with a $(0)$ to recover the type of $F$. If $F$ does not have a component of type $(0)$, then $F$ has type $(1, a_1, a_2, \ldots, a_k)$.

Suppose the \sdeck of $F$ has another type $(0,b_1, b_2, \ldots, b_k, *)$ where 
$\{a_i : 1 \leq i \leq k\} \neq \{b_i : 1 \leq i \leq k\}$ as multisets. Then $F$ has a 
component of type $(0)$, since otherwise the zeros in the \sdeck types come from a $(1)$ 
of $F$ and all such types must be the same up to order of entries.

Thus we assume $(0,a_1, a_2, \ldots, a_k, *)$ is the only type in the \sdeck with a 0. If the 
type of $F$ had a 0 and two distinct nonzero numbers then the \sdeck contains two distinct types 
with a 0. Hence we may assume that
\[ 
	(0,a_1, a_2, a_3, \ldots, a_k, *) 
	= (0, a-1, a, a, \ldots, a, *) \quad \text{ for some $a \geq 2$}.	
\]
The type of $F$ is either $(0,a,a,a,\ldots,a)$ or $(1,a-1,a,a,\ldots,a)$. 
Lemma \ref{lem:treetypes} (iii) implies that in the latter case the \sdeck has 
$(1,a-1,a-1,a,\ldots,a,*)$, whereas the former case cannot have the same type. 
Thus we can decide the type of $F$ from the \sdeck in Case 3.

In all cases we can decide the type of $F$ from the degree sequence and the type-deck of $F$. Therefore Lemma \ref{lem:forestType} holds and Lemma \ref{lem:recon3} follows.
\end{proof}

\section{Applications}\label{sec:appl}
In this section we prove some facts about Euler polynomials, the number of
dominating sets, and sums over spanning forests of a graph. The common theme
of these results is a very simple way of proving by just using different
representations of the bipartition polynomial.

We denote by $\mathcal{F}(G)$ the set of spanning forests of the graph $G$.
\begin{theorem}
	The Euler polynomial of a graph $G$ of order $n$ and size $m$ satisfies
	\[ 
		\mathcal{E}(G,z) = (1+z)^{m-n}(-z)^n \sum_{H\in\mathcal{F}(G)}
		\left(-\frac{1+z}{z}\right)^{k(H)}
		\left(\frac{1-z}{1+z}\right)^{\mathrm{ext}(H)}.
	\]
\end{theorem} 
\begin{proof}
	We use the forest representation of the bipartition polynomial that is given in 
	Theorem \ref{theo:forest_representation} and Equation (\ref{eq:euler_bip}); we obtain 
	\begin{eqnarray*}
		\mathcal{E}(G,z) &=& \frac{(1+z)^m}{2^n}B\left(G;1,1,\frac{-2z}{1+z} \right) \\
		&=& \frac{(1+z)^m}{2^n} \hspace{-5pt}\sum_{H\in\mathcal{F}(G)}
			\hspace{-7pt} 2^{\mathrm{iso}(H)} \left(\frac{-2z}{1+z}\right)^{n-k(H)}
			\left(\frac{1-z}{1+z}\right)^{\mathrm{ext}(H)} \hspace{-3pt} 
			2^{k(H)-\mathrm{iso}(H)} \\
		&=& (1+z)^{m-n}(-z)^n \sum_{H\in\mathcal{F}(G)}
		\left(-\frac{1+z}{z}\right)^{k(H)}
		\left(\frac{1-z}{1+z}\right)^{\mathrm{ext}(H)}.
		\end{eqnarray*}
	For the second equality, we used the simple relation 
	$|\mathrm{Comp}(H)| + \mathrm{iso}(H) = k(H)$.
\end{proof}

The next theorem provides a representation of the Euler polynomial as a sum ranging over
subsets of the vertex set.
\begin{theorem}
	The Euler polynomial of a graph $G=(V,E)$ satisfies
	\[ 
	\mathcal{E}(G,z) = \frac{(1+z)^{|E|}}{2^{|V|}}\sum_{W\subseteq V}
	\left(\frac{1-z}{1+z}\right)^{|\partial W|}.
	\]
\end{theorem}
\begin{proof}
	The result can be obtained via the multiplicative representation of the bipartition
	polynomial according to Theorem \ref{theo:prod_representation}. The substitution
	of this representation for the bipartition polynomial in Equation (\ref{eq:euler_bip})
	yields
	\begin{eqnarray*}
		\mathcal{E}(G,z) &=& \frac{(1+z)^m}{2^n}B\left(G;1,1,\frac{-2z}{1+z} \right) \\
		&=& \frac{(1+z)^m}{2^n} \sum_{W\subseteq V}  \prod_{v\in N_{G}(W)}
		\left[\left(1-\frac{2z}{1+z}\right)^{|\partial v\cap \partial W|} \right] \\
		&=& \frac{(1+z)^m}{2^n} \sum_{W\subseteq V}  
		\left(\frac{1-z}{1+z}\right)^{|\partial W|}.
	\end{eqnarray*}
	The last equality follows from the fact that $|N_{G}(v)\cap W|$ is the number of edges 
	that connect $v$ to a vertex of $W$.
	Hence when we take the product over all vertices in $N_G(W)$, then we count each edge
	in $\partial W$ exactly once.
\end{proof}

The following statement can be proven also via the principle of inclusion--exclusion.
However, our knowledge about the bipartition polynomial offers an even faster way of proof.
\begin{theorem}
	The number $d(G)$ of dominating sets of a graph $G$ satisfies
	\[ 
		d(G) = 2^n \sum_{W\subseteq V}(-1)^{|W|}
		\left(\frac{1}{2}\right)^{|N_G[W]|}.
	\]
\end{theorem}
\begin{proof}
	Here we use the product representation of the bipartition polynomial
	of a simple graph. The restriction to simple graphs does not change the domination
	polynomial. 
	According to Equation (\ref{eq:dom_bip}), we obtain
	\begin{align*}
		d(G) &= D(G,1) \\
		&= 2^n B\left(G;-\frac{1}{2},\frac{1}{2},-1\right) \\
		&= 2^n \sum_{W\subseteq V}\left(-\frac{1}{2}\right)^{|W|}
		\prod_{v\in N_G(W)}
		\left[\left(\frac{1}{2}\right)\left[0^{|N_{G}(v)\cap W|}-1 \right]  +1 \right]\\
		&= 2^n \sum_{W\subseteq V}\left(-\frac{1}{2}\right)^{|W|}
		\prod_{v\in N_G(W)}
		\left(\frac{1}{2}\right) \\
		&= 2^n \sum_{W\subseteq V}\left(-\frac{1}{2}\right)^{|W|}
		\left(\frac{1}{2}\right)^{|N_G(W)|},
	\end{align*} 
which is equivalent to the statement of the theorem.
\end{proof}

\begin{theorem}
	Let $G=(V,E)$ be a graph with a linearly ordered edge set and $\mathcal{F}_0(G)$ 
	the set of all spanning forests of $G$ with external activity 0. Then
	\[ 
		\sum_{H\in\mathcal{F}_0(G)}(-2)^{k(H)} = (-1)^n 2^{k(G)}.
	\]
\end{theorem}
\begin{proof}
	The statement follows immediately from the forest representation of the bipartition 
	polynomial that is given in Theorem \ref{theo:forest_representation} by substituting
	$x=1$, $y=1$, $z=-1$ in $B(G;x,y,z)$.
\end{proof}

\begin{theorem}\label{theo:biolor}
	Let $G=(V,E)$ be a simple undirected graph with $n$ vertices. The number of 
	bicolored subgraphs of $G$ with exactly $i$ isolated vertices and exactly $j$ edges 
	is given by the coefficient of $x^iz^j$ in the polynomial
	\[ 
	(2x-1)^n B\left(G;\frac{1}{2x-1},\frac{1}{2x-1},z\right).
	\]
\end{theorem}
\begin{proof}
	An edge subset $F\subseteq E$ of $G$ induces a subgraph that can be properly
	colored with two colors if and only if $(V,F)$ is a bipartite graph. The number
	of bicolored graphs with edge set $F$ is then given by $2^{k(V,F)}$. Substituting
	$x$ and $y$ by $1/(2x-1)$ in $B(G;x,y,z)$ and multiplying the resulting expression
	with $(2x-1)^n$ yield
	\[ 
	(2x-1)^n \hspace{-14pt}\sum_{\substack{F\subseteq E \\(V,F)\text{ is bipartite}}}
	\hspace{-20pt} z^{|F|}\left(\frac{2x}{2x-1}\right)^{\mathrm{iso}(V,F)} 
	\hspace{-20pt} \prod_{(U,A)\in\mathrm{Comp}(V,F)}
	\hspace{-5pt}\left(\frac{2}{2x-1}\right)^{|U|}.
	\]
	Using the equations $\mathrm{iso}(V,F)+|\mathrm{Comp}(V,F)|=k(V,F)$ and
	\[ 
		\mathrm{iso(V,F)} + \sum_{(U,A)\in\mathrm{Comp}(V,F)} |U| = n,
	\]
	we obtain
	\[ 
	(2x-1)^n B\left(G;\frac{1}{2x-1},\frac{1}{2x-1},z\right)
	= \sum_{\substack{F\subseteq E \\(V,F)\text{ is bipartite}}}
	\hspace{-20pt} 2^{k(V,F)} x^{\mathrm{iso}(V,F)} z^{|F|},
	\]
	which proves the statement.
\end{proof}

\begin{corollary}
	Let $G$ be a simple graph of order $n$. The number of bicolored subgraphs of $G$
	without any isolated vertices is
	\[ 
		(-1)^nB(G;-1,-1,1).
	\]
\end{corollary}
\begin{proof}
	This follows immediately from the last line of the proof of Theorem \ref{theo:biolor}
	by substituting $x=0$ and $z=1$.
\end{proof}

\section{Conclusions and Open Problems}
The bipartition polynomial emerges as a powerful tool for proving equations in graphical
enumeration. It shows nice relations to other graph polynomials, offers a couple of
useful representations, and is polynomially reconstructible. However, there are still many
open questions for the bipartition polynomial; we consider the following ones most
interesting:
\begin{itemize}
	\item Equation (\ref{eq:planar}) gives a relation between the bipartition polynomial
	of a planar graph and its dual. However, this equation is restricted to the evaluation
	of $B(G;x,y,z)$ at $x=y=1$. Is there a way to generalize this result?
	\item The Ising and matching polynomial of a graph $G$ can be derived from the 
	corresponding polynomials of the complement $\bar{G}$. Can we calculate the bipartition
	polynomial of a graph from the bipartition polynomial of its complement?
	\item There are known pairs of nonisomorphic graphs with the same bipartition
	polynomial. However, despite all efforts by extensive computer search, we could not
	find a pair of nonisomorphic trees with coinciding bipartition polynomial. We know
	that no such pair for trees with order less than 19 exists. Is the bipartition
	polynomial able to distinguish all nonisomorphic trees?
\end{itemize}

\section{Acknowledgments}
We thank Jo Ellis-Monaghan, Andrew Goodall, Johann A. Makowsky, and Iain Moffatt -- 
the organizers of the Dagstuhl seminar \emph{Graph Polynomials: Towards a 
Comparative Theory} (2016). This excellent and fruitful workshop initiated the cooperation of
the authors of this paper.

\printbibliography

\end{document}